\documentclass[a4paper,10pt,leqno,twoside]{amsart}

\usepackage[english]{babel} \usepackage{amsmath, amssymb, latexsym,
  epic, epsfig, rotating, fancyhdr, amsthm, pifont, enumitem}

\usepackage{amscd,mathrsfs}
\usepackage[ansinew]{inputenc}
\usepackage{graphicx,psfrag}

\newcommand{\ld}{\ensuremath{,\ldots,}}
\newcommand{\ssq}{\ensuremath{\subseteq}}
\newcommand{\smin}{\ensuremath{\setminus}}

\newcommand{\inte}{\ensuremath{\mathrm{int}}}

\newcommand{\conv}{\ensuremath{\mathrm{Conv}}}

\newcommand{\torus}{\ensuremath{\mathbb{T}^2}}

\newcommand{\homeo}{\ensuremath{\mathrm{Homeo}}}
\newcommand{\homtwo}{\ensuremath{\mathrm{Homeo}_0(\mathbb{T}^2)}}

\newcommand{\nfolge}[1]{\ensuremath{(#1)_{n\in\mathbb{N}}}}

\newcommand{\alphlist}{\begin{list}{(\alph{enumi})}{\usecounter{enumi}\setlength{\parsep}{2pt}
      \setlength{\itemsep}{1pt} \setlength{\topsep}{5pt}
      \setlength{\partopsep}{3pt}}}
\newcommand{\arablist}{\begin{list}{(\arabic{enumi})}{\usecounter{enumi}\setlength{\parsep}{2pt}
          \setlength{\itemsep}{1pt} \setlength{\topsep}{5pt}
          \setlength{\partopsep}{3pt}}}
\newcommand{\romanlist}{\begin{list}{(\roman{enumi})}{\usecounter{enumi}\setlength{\parsep}{2pt}
              \setlength{\itemsep}{1pt} \setlength{\topsep}{5pt}
              \setlength{\partopsep}{3pt}}}
\newcommand{\Romanlist}{\begin{list}{(\Roman{enumi})}{\usecounter{enumi}\setlength{\parsep}{2pt}
              \setlength{\itemsep}{1pt} \setlength{\topsep}{5pt}
              \setlength{\partopsep}{3pt}}}
\newcommand{\bulletlist}{\begin{list}{$\bullet$}{\setlength{\parsep}{2pt}
                \setlength{\itemsep}{1pt} \setlength{\topsep}{5pt}
                \setlength{\partopsep}{3pt}\setlength{\leftmargin}{15pt}}} 
\newcommand{\Alphlist}{\begin{list}{(\Alph{enumi})}{\usecounter{enumi}\setlength{\parsep}{2pt}
      \setlength{\itemsep}{1pt} \setlength{\topsep}{5pt}
      \setlength{\partopsep}{3pt}}}
 \newcommand{\listend}{\end{list}}

\newcommand{\T}{\ensuremath{\mathbb{T}}}

\newcommand{\N}{\ensuremath{\mathbb{N}}} 
\newcommand{\R}{\ensuremath{\mathbb{R}}}
\newcommand{\Z}{\ensuremath{\mathbb{Z}}}

\newcommand{\cA}{\mathcal{A}}
\newcommand{\cB}{\mathcal{B}}
\newcommand{\cC}{\mathcal{C}}

\newcommand{\cF}{\mathcal{F}}

\newcommand{\cM}{\mathcal{M}}

\newcommand{\cO}{\mathcal{O}}

\newcommand{\cR}{\mathcal{R}}
\newcommand{\cS}{\mathcal{S}}

\newcommand{\nLim}{\ensuremath{\lim_{n\rightarrow\infty}}}
\newcommand{\iLim}{\ensuremath{\lim_{i\rightarrow\infty}}}

\newcommand{\inergsum}{\ensuremath{\sum_{i=0}^{n-1}}}

\newcommand{\ntel}{\ensuremath{\frac{1}{n}}}

\newcommand{\achtel}{\ensuremath{\frac{1}{8}}}

\newtheoremstyle{tobthm}{3pt}{3pt}{\itshape}{0pt}{\bfseries}{.}{0.5eM}{}
\theoremstyle{tobthm}

\newtheorem{definition}{Definition}[section]
\newtheorem{thm}[definition]{Theorem}

\newtheorem{lem}[definition]{Lemma}
\newtheorem{lemma}[definition]{Lemma}
\newtheorem{cor}[definition]{Corollary}
\newtheorem{prop}[definition]{Proposition}

\newtheorem{conj}[definition]{Conjecture}

\newtheoremstyle{tobrem}{3pt}{3pt}{\normalfont}{0pt}{\bfseries}{.}{0.5em}{}
\theoremstyle{tobrem}

\numberwithin{equation}{section}
\numberwithin{figure}{section}

\def\Cl{\mathop\mathrm{Cl}}

\title{\Large\textsc{Rotation sets and almost periodic sequences}}

\author{T.~J\"ager$^\ast$, A.~Passeggi$^\dagger$ and S.~\v Stimac$^§$}

\thanks{$^\ast$ Department of Mathematics, TU Dresden, {\tt
    Tobias.Oertel-Jaeger@tu-dresden.de} } \thanks{$^\dagger$ Department of
  Mathematics, TU-Dresden. {\tt alepasseggi@gmail.com}.}  \thanks{$^§$
  Department of Mathematics, University of Zagreb \& Department of Mathematical
  Sciences IUPUI, {\tt sonja@math.hr}}

\begin{document}

\begin{abstract}
  We study the rotational behaviour on minimal sets of torus homeomorphisms and
  show that the associated rotation sets can be any type of line segments as
  well as non-convex and even plane-separating continua. This shows that
  restrictions which hold for rotation set on the whole torus are not valid on
  minimal sets.

  The proof uses a construction of rotational horseshoes by Kwapisz to transfer
  the problem to a symbolic level, where the desired rotational behaviour is
  implemented by means of suitable irregular Toeplitz sequences.
\end{abstract}

\maketitle

\section{Introduction.}

Given a torus homeomorphisms $f:\torus \to \torus$ homotopic to the identity, 
a lift $F:\R^2\to\R^2$ and any set $M\ssq \torus$, the {\em rotation set of
$F$ on $M$} is defined as 
\begin{equation}
  \label{e.RotationSet}
  \rho_M(F) \ = \ \left\{\rho\in\R^2\ \left|\ \exists n_i\nearrow\infty,\ 
  z_i\in\pi^{-1}(M): \iLim \left(F^{n_i}(z_i)-z_i \right)/n \ = \ \rho \right.\right\} \ ,
\end{equation}
where $\pi:\R^2\to\torus$ denotes the canonical projection.  In case $M=\torus$,
the set $\rho(F)=\rho_{\torus}(F)$ is simply called the {\em rotation set
  of $F$}. It takes a central place in the classification of torus
homeomorphisms, since a wealth of dynamical information can be obtained from the
shape of $\rho(F)$ (see, for example,
\cite{franks:1989}--\nocite{llibre/mackay:1991,franks:1995,jaeger:2009b,Davalos2012THwithRotationSegments}\cite{KoropeckiTal2012StrictlyToral}
and references therein). A crucial fact in this context is that $\rho(F)$ is
always compact and convex \cite{misiurewicz/ziemian:1989}. 
Concerning the rotational behaviour on minimal subsets, it is known that if
$\rho(F)$ has non-empty interior, then for every vector $\rho\in\inte(\rho(F))$
there exists a minimal set $M_\rho\ssq\torus$ with $\rho_{M_\rho}(F)=\{\rho\}$
\cite{misiurewicz/ziemian:1991}. Further, if $M$ is minimal, then $\rho_M(F)$ is
always compact and connected \cite{PasseggiRotSetAxA}, and examples in
\cite{PasseggiRotSetAxA} show that it can be a line segment of the form
$\{0\}\times[a,b]$ with $a<b$.

The aim of this note is to explore more complex rotational behaviour on minimal
sets. The bottomline is that apparently no restrictions exist for the associated
rotation sets, besides compactness and connectedness. We demonstrate this by
means of three types of examples, which are actually all realised by the same
torus homeomorphism.  Denote by $\homtwo$ the set of homeomorphisms of \torus\
homotopic to the identity. 
\begin{thm} \label{t.mainresult}
  There exists $f\in\homtwo$ with lift $F:\R^2\to\R²$ such that 
\begin{itemize}
\item[(a)] for an open set $V\ssq\R^2$ and all $v\in V$ there is a minimal set
  $M_v$ such that $\rho_{M_v}(F)$ is a line segment of positive length contained
  in $v+\R v^\perp$;
\item[(b)] for some minimal set $M$, the associated rotation set $\rho_M(F)$ is
  plane-separating;
\item[(c)] for some minimal set $M'$, the associated rotation set $\rho_{M'}(F)$
  has non-empty interior.
\end{itemize}
\end{thm}
The proof of Theorem~\ref{t.mainresult} can roughly be outlined as follows. The
homeomorphism $f$ is chosen such that it has a {\em rotational horseshoe} with
three symbols and the topology depicted in Figure~\ref{fig.RotationalHorseshoe}.
This construction essentially goes back to \cite{kwapisz:1992}, where it is
implemented in much greater generality to show that every rational polygon can
occur as the rotation set of a torus homeomorphism. For our purposes, the
important fact is that in this situation we obtain an invariant set
$\Lambda=\bigcap_{n\in\Z}f^n(\pi(D))$, where $D\ssq\R^2$ is a topological disk
that projects injectively to $\torus$, and a symbolic coding
$h:\Lambda\to\{0,1,2\}^\Z$ such that $h\circ f=\sigma\circ h$. Moreover, given
$z\in\Lambda$, the entry $h(z)_0$ determines whether a lift $\hat z\in D$ of
$z\in\pi(D)$ remains in $D$, moves to $D+(1,0)$ or to $D+(0,1)$. Consequently,
if we let $v_0=(0,0)$, $v_1=(1,0)$ and $v_2=(0,1)$, then the displacement vector
$F^n(z)-z$ differs from the vector $\inergsum v_{h(z)_i}$ only by an error term
that is bounded uniformly in $n\in\N$ and $z\in\Lambda$. Asymptotically, this
means that rotation vectors and sets are completely determined by the coding,
and the rotational behaviour on minimal sets can be studied on a purely symbolic
level.  The crucial issue on the technical side then is to construct suitable
almost periodic sequences that produce the desired rotation sets. To that end,
we work within the class of irregular Toeplitz sequences, which have been used
previously to produce a number of interesting examples in topological and
symbolic dynamics
\cite{JacobsKeane1969ToeplitzSequences,MarkleyPaul1979PositiveEntropyToeplitzFlows,Williams1984ToeplitzFlows,Downarowicz2005ToeplitzFlows}. In
certain aspects, our construction is reminiscent of these more classical ones.

\begin{figure}
  \centering
  \epsfig{file=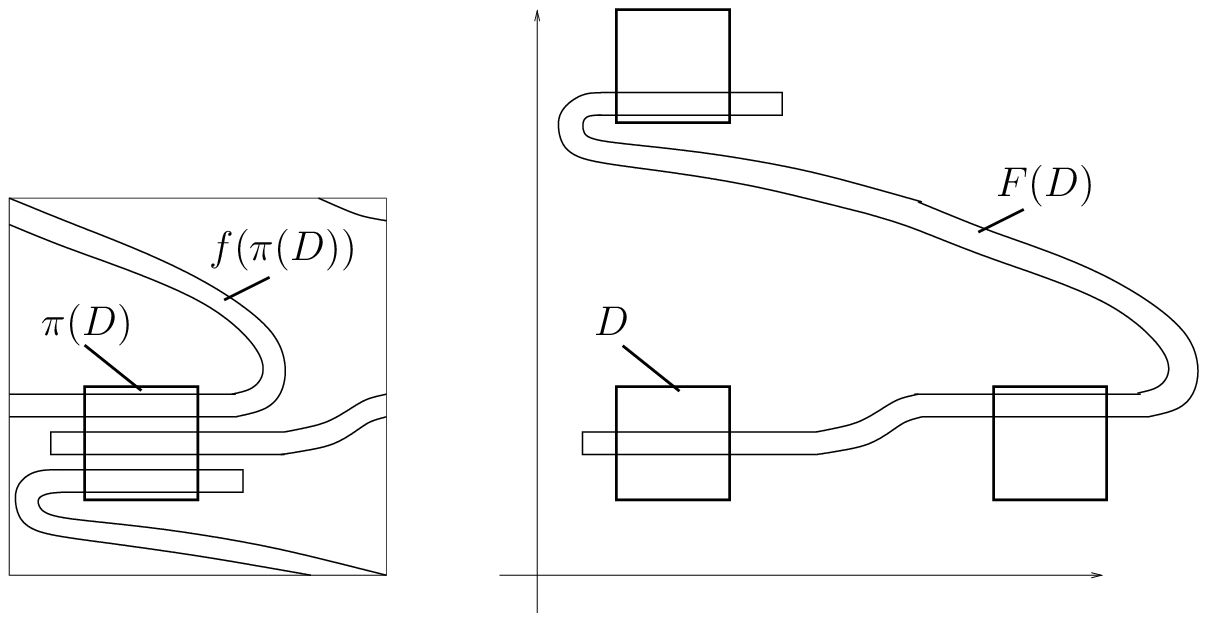, clip=,width=\linewidth}

  \caption{Geometry of a rotational horseshoe: The horseshoe is located in the
    topological disk $\pi(D)\ssq \T^2$ shown on the left. The Markov partition is
    given by the preimages of the connected components of $\pi(D)\cap f(\pi(D))$. The
    situation for the lift is depicted on the right, the displacement vectors
    are $v_0=(0,0),\ v_1=(1,0)$ and $v_2=(0,1)$.\label{fig.RotationalHorseshoe}}
\end{figure}

It is well-known that a dynamical situation like the one in
Figure~\ref{fig.RotationalHorseshoe} is stable under perturbations. Hence, our
construction immediately yields an open set of torus homeomorphisms that satisfy
the assertions of Theorem~\ref{t.mainresult}. Moreover, it is known that the
existence of rotational horseshoes is $\cC^0$-generic within an open and dense
subset of \homtwo, see \cite{PasseggiRotSetAxA}. In order to give a precise
statement in our context, we denote by $\cF$ the set of those $f\in\homtwo$
whose rotation sets have non-empty interior. Then $\cF$ is open in the
$\cC^0$-topology \cite{misiurewicz/ziemian:1991}, and we have
\begin{thm}\label{t.genericity}
  The set of $f\in\homtwo$ which satisfy the assertions of
  Theorem~\ref{t.mainresult} form an open and dense subset of $\cF$.
\end{thm}
In fact, we believe that this set is equal to $\cF$ and, that arbitrary continua
in the interior of $\rho(F)$ can be realised. This leads to the following
\begin{conj}
  Given $f\in\homtwo$ with $\inte(\rho(F))\neq\emptyset$ and any continuum
  $C\ssq\inte(\rho(F))$, there exists a minimal set $M_C$ such that
  $\rho_{M_C}(F)=C$.
\end{conj}
\medskip

\noindent{\bf Acknowledgments.} This work was partially carried out during the
conference `Dimension Theory and topological complexity of skew products' in
Vienna in October 2013, and we would like to thank the organisers Henk Bruin and
Roland Zweim\"uller for creating this opportunity as well as the
Erwin-Schr\"odinger-Institute for its hospitality and the superb conditions
provided during the event. T.J. acknowledges support of the German Research
Council (Emmy Noether Grant Ja 1721/2-1) and thanks the Max-Planck-Institute for
Mathematics in Bonn for its hospitality during the Dynamics and Numbers activity
in June and July 2014, when this work was finalised. S.\v S. acknowledges
partial support of the NEWFELPRO Grant No.~24 HeLoMa. 

\section{Rotational horseshoes and the symbolic computation of rotation
  sets}\label{symbolicdynamics}

{\em Rotational horseshoes.} We say that $R\subset \T^2$ is a {\em (topological)
  rectangle} if it is homeomorphic to the unit square $[0,1]^2$. Given an
invariant set $C\subset\T^2$ of $f\in\homeo_0(\T^2)$, we say that a family of
pairwise disjoint rectangles $\mathcal{R}=\{R_0,\dots,R_N\}$ is a {\em
  partition} of $C$ if $C\subset \bigcup_{i=0}^N R_i$. In this case, we let
$\Sigma:=\{0,\dots,N\}^\Z$ and denote by $\mathcal{S}$ the set of those
sequences $\omega\in\Sigma$ for which there exists $x\in C$ with $f^i(x)\in
R_{\omega(i)}$ for all $i\in \Z$. Then $\mathcal{S}$ is compact and invariant
under the shift $\sigma$ on $\cS$.  If it happens, as in the classical horseshoe
construction, that for every sequence $\xi\in \mathcal{S}$ there is a unique
$x\in C$ with $f^i(x)\in R_{\omega(i)}$ for all $i\in\Z$, then the map
$h_{\mathcal{R}}:\mathcal{S}\to C$ taking $\xi$ to the corresponding $x$ is a
conjugacy from $\sigma_{|\cS}$ to $f_{|C}$. This happens to be the case for any
zero-dimensional hyperbolic set $C$ in $\T^2$ with local product structure. In
fact, in this situation the partition $\cR$ can be chosen such that it is a
\textit{Markov partition}, that is, $\cS$ is a subshift of finite type (see
\cite{PasseggiRotSetAxA,beguin:2002}). If $\cS=\Sigma$, we call $C$ a horseshoe
and say it is {\em rotational} if in addition the following two
properties are satisfied: (i) there exists a bounded topological disk
$D\ssq\R^2$ such that $\pi(D)\ssq\torus$ is a topological disk and
$\bigcup_{i=0}^NR_i\ssq \pi(D)$; (ii) for each $i=0\ld N$ there exists a unique
vector $v_i$ such that if $z\in D\cap \pi^{-1}(C)$ and $\pi(z)\in R_i$, then
$F(z)\in D+v_i$. In other words, in a rotational horseshoe the symbolic coding
determines to which copy of $D$ a point is mapped by $F$. As mentioned before,
this allows to compute rotation sets and rotation vectors on a purely symbolic
level.

More precisely, given a finite word $w=w_1\dots w_m$, let $|w|=m$ be the length
of $w$ and $\psi(w)=\sum_{j=1}^{m}v_{w_j}$. Further, for a closed and
$\sigma$-invariant set $\mathcal{M}\subset \cS$ we define
\begin{equation}
  \rho_{\mathcal{M}}\ =\ \left\{\left.\nLim \frac{\psi(w^{(n)})}{|w^{(n)}|} \ \right| \ w^{(n)}
    \mbox{ is a subword of some }\omega^{(n)}\in\mathcal{M}
    \textrm{ and }  |w^{(n)}| \geq n \right\} \ .
\end{equation}
The following lemma provides the crucial estimate that allows to translate these
symbolic to dynamical rotation sets. Given $\omega\in\Sigma$, we let
$\omega_{[1,n]}=\omega(1)\omega(2)\ldots\omega(n)$.
\begin{lemma}[\cite{PasseggiRotSetAxA}, Proposition 2.1]\label{SymbDesp} There
  exists $r>0$ so that for any $z\in C$ we have $F^{n+1}(z)-z\in
  B_r(\psi(h_{\cR}^{-1}(z)_{[1,n]}))$.
\end{lemma}
As a direct consequence, we obtain
\begin{cor} \label{c.symbolic_rotset}
  $\rho_{\mathcal{M}}=\rho_{h_{\mathcal{R}}(\mathcal{M})}(F)$.
\end{cor}
A sequence $\omega\in\Sigma$ is {\em almost periodic} if any finite subword
occurs infinitely often and the time between two occurrences is uniformly
bounded. It is well-known that $\omega$ is almost periodic if and only if
$\overline{\cO_\sigma(\omega)}$ is minimal. Moreover, in this case
$\overline{\cO_\sigma(\omega)}$ coincides with the set of those sequences
$\xi\in\Sigma$ which have exactly the same subwords as $\omega$
\cite{auslander1988minimal}. Together with 
Corollary~\ref{c.symbolic_rotset}, this yields the following statement.
\begin{prop}\label{propextension}
  Given an almost periodic sequence $\omega\in\Sigma$, the set
  $\cM=\overline{\cO_f(h(\omega))}$ is minimal with respect to $f$ and we have 
  \begin{equation} 
    \rho_{\cM}(F)=\left\{\left.\nLim
      \frac{\psi(w^{(n)}|}{|w^{(n)}|}\ \right|
      w^{(n)} \textrm{ is a subword of } \omega, |w^{(n)}|\geq n \right\} \ .
\end{equation}
\end{prop}
For constructing suitable almost periodic sequences, it is convenient to work
only in the one-sided shift space. Due to the following folklore lemma, this is
sufficient. 
\begin{lem}
  Suppose $\omega^+\in\Sigma^+$ is almost periodic and $\omega$ is any sequence
  in $\Sigma$ whose right side coincides with $\omega^+$. Then
  $\overline{\cO_\sigma(\omega)}$ is minimal and coincides with the set of sequences that 
  have exactly the same finite subwords as $\omega^+$.
\end{lem}
A particular case of almost periodic sequences are Toeplitz sequences. A
sequence $\omega^+\in\Sigma^+$ ($\xi\in\Sigma$) is called a {\em Toeplitz
  sequence} if for every $j\in\N$ ($j\in\Z$) there exists $p\in\N$ so that
$\omega^+_{j+np}=\omega^+_j$ for all $n\in\N$ ($\omega_{j+np}=\omega_j$ for all
$n\in\Z$). In other words, every entry of a Toeplitz sequence occurs
periodically. However, since the periods depend on the position, the sequence
itself need not be periodic. In fact, aperiodicity is often included in the
definition, and we will follow this convention here.

\section{Realisation of rotation sets by Toeplitz sequences }

\subsection{Preliminary notions}

We fix $f\in\homtwo$ such that $f$ has a rotational horseshoe $C$ with three
symbols and displacement vectors $v_0=(0,0),\ v_1=(1,0)$ and $v_2=(0,1)$, as in
Figure~\ref{fig.RotationalHorseshoe}. Thus, there exists a bounded topological disk $D\ssq\R^2$ and a
partition $\mathcal{R}=\{R_0,R_1,R_2\}$ of $C$ with $\bigcup_{i=0}^2 R_i\ssq
\pi(D)$ such that $F(\pi^{-1}(R_i)\cap D) \ssq D+v_i$. As before, we denote by
$h_{\cR}$ the conjugacy between the shift $\sigma$ on $\Sigma:=\{0,1,2\}^{\Z}$
and $f_{|C}$. As we will see in Section~\ref{Abundance}, the family of such maps
is open and dense in the set $\cF\ssq\homtwo$ of torus homeomorphisms with
non-empty interior rotation sets. According to Corollary~\ref{c.symbolic_rotset}
and Proposition~\ref{propextension}, our aim is to construct almost periodic sequences whose
associated rotation sets are line segments of positive length, separate the
plane or have non-empty interior. To that end, we first introduce a general
block structure which produces Toeplitz sequences through an inductive
construction. \smallskip

{\em A general block structure.}  Suppose $(b_n)_{n\in\N}$ and $(d_n)_{n\in\N}$
are sequences of positive integers, with $d_{n+1}$ a multiple of $d_n$ for all
$n\in\N$. Let $a_1\in\N$. Slighly abusing notiation, we denote by $[k,l]$ the
interval of all integers $i$ with $k\leq i \leq l$, similarly for open and
half-open intervals.  Then we recursively define
\begin{itemize}
\item $a_{n+1}=(b_nd_n+1)a_n$
\item $\cA_n=[1,a_n]+d_na_n\N$
\item $\cB_n=\bigcup_{i=1}^n \cA_n$
\item $\cC_n=\cB_n\setminus \cB_{n-1}$
\end{itemize}
We call the  maximal intervals in $\cA_n$ \textit{blocks of level $n$}. If such
a block is not equal to the first block $[1,a_n]$, we call it a {\em repeated
  block}.  The following facts are easy to check. 
\begin{itemize}
\item[(F1)] Given $k<k'$, any block of level $k'$ starts and ends with a block of
level $k$.
\item[(F2)] If two blocks of levels $k$ and $k'$ are disjoint and $k\leq k'$,
  then the interval between the blocks has length $\geq (d_k-1)a_k$.
\item[(F3)] The asymptotic density of $\cB_n$ is at most $\delta_n=\sum_{j=1}^n\frac{1}{d_j}$.
\item[(F4)] If $J$ is an interval of integers whose length is a multiple of
  $a_nd_n$, then $\frac{1}{\mid J\mid }|J\cap \cB_n|\leq
  \delta_n$. Consequently, given $M\in\N$ and any interval $J'$ of length $\geq
  a_nd_n/M$ we have $\frac{1}{\mid J'\mid}|J'\cap \cB_n|\leq M\delta_n$. Here
  $|J|$ denotes the cardinality of a set $J\ssq\N$.
\item[(F5)] If a sequence $\omega=(a_i)_{i\in\N}$ is chosen so that for all $n\in\N$, $j\in[1,a_n]$ and $k\in\N$
it satisfies $a_{j+ka_nd_n}=a_j$, then $\omega$ is Toeplitz.
\end{itemize}
We let $\delta_{\infty}=\lim_n\delta_{n}=\sup_n\delta_{n}$.
\smallskip

\subsection{Line segments}

We first need to specify the open set $V\ssq \R^2$ in
Theorem~\ref{t.mainresult}. In principle, we could take the whole interior of
the simplex $\Delta$ spanned by the vectors $v_0,v_1$ and $v_2$ defined
above. However, for the sake of convenience we let $\alpha = \alpha(v) :=
\left\langle v_1, \frac{v}{\|v\|}\right\rangle,\ \beta=\beta(v) := \left\langle
  v_2, \frac{v}{||v||} \right\rangle$ and define $V$ as the subset of vectors in
$\Delta$ for which $\|v\|\leq\min\{\alpha,\beta\}$, which will simplify our
construction below to some extent.

Given $\omega^+\in\Sigma^+$, we denote by $\cM(\omega^+)=\Omega(\omega)$ the
omega-limit set of a sequence $\omega\in\Sigma$ whose right side coincides with
$\omega^+$. According to Proposition~\ref{propextension}, $\cM(\omega)$ is a
minimal set, and the subwords of sequences in $\cM(\omega)$ are exactly the
subwords of $\omega^+$. Given $v\in V$, our aim is now to construct a one-sided
sequence $\omega_v= (\omega_v(j))_{j \in \N}$ such that $\rho_{\cM(\omega_v)}$
is a line segment of positive length contained in $v+\R v^\perp$. To that end,
we use the above general block structure with the following specifications. We
let $b_n=1$ and $d_n = 2^{n + t}$ for some integer $t$ such that $\delta_\infty
\leq \frac{\|v\|}{10\max\{\alpha,\beta\}}$. We start the construction
with an integer $a_1\geq 2M/\|v\|+1$. where
$M=\|v\|+\max\{\alpha,\beta\}$. Further, we let $D(l,j)= \left\langle
  \sum_{i=l}^j v_{\omega_v(i)},\frac{v}{\|v\|}\right\rangle-(j-l+1)\|v\|$. The
sequence $\omega_v$ will be constructed by induction on the sets $\cC_n$. To
that end, we first define $\omega_v$ on $[0,a_n]\cap \cC_n$ and then extend it
to the whole of $\cC_n$ by $a_nd_n$-periodic repetition. On $[0,a_n]$, we choose
the entries $\omega_v(j)$ by induction on $j$ according to the following rules.
\begin{itemize}
\item[(I)] If $n$ is odd, we let $\iota=1$, if $n$ is even we let $\iota=2$.
\item[(II)] If neither of $j,j+1\ld j+K$ intersects a block of level $<n$, then
  we choose $\omega_v(j)\in\{0,\iota\}$ such that $D(1,n)$ is contained
  in the interval $[0,M]$. If this is true for both possible choices $0$ and
  $\iota$, we let $\omega_v(j)=0$.
\item[(III)] If $B=[m+1,m+a_k]\ssq [0,a_n]$ is a block of level $k<n$ which is
  not contained in a larger block of level $<n$, then we choose
  $\omega_v(m-K+1)\ld\omega_v(m)$ such that $D(1,j)\in [-M,M]$ for all $j=m-K+1\ld
  m$ and $D(1,m)\in [-D(1,a_k),M-D(1,a_k)]$. In order to make this choice unique,
  we require in addition that $D(1,j)$ always takes the smallest value which is
  possible under these conditions. This means we put 0 whenever possible, and
  $\iota$ only when necessary.
\end{itemize}
In order to see that these rules are consistent, note that if $\omega_v(j)=0$,
then $D(1,j)=D(1,j-1)-\|v\|$, if $\omega_v(j)=1$ then $D(1,j)=D(1,j-1)+\alpha$
and if $\omega_v(j)=2$ then $D(1,j)=D(1,j-1)+\beta$. In each step, we therefore
have the choice to either increase or decrease the value of $D(1,j)$. Thus, if
$D(1,j-1)\in[0,M]$, then due to the choice of $M$ we can always choose
$\omega_v(j)$ in such a way that $D(1,j)\in[0,M]$ as well. Since rule (III)
ensures that $D(1,j-1)\in[0,M]$ whenever $j-1$ is the end of a block of level
$<n$, it is possible to follow rule (II) whenever it applies. If $j=m-K+1$,
where $m+1$ is the starting point of a block of level $<n$ and $D(1,j-1) \in
[0,M]$, then choosing $\omega_v(i)=0$ for all $i=j\ld m$ would yield $D(m)\leq
M-K\|v\|\leq -M$. Thus, by replacing some of the zeros with $\iota$'s, it is
also possible to meet the requirements of rule (III). Note here that due to the
choice of $K=a_1-1$ and the spacing of the blocks, the integers $j\ld m$ are
not contained in any repeated block of level $<n$. Altogether, this implies that
the above algorithm yields a well-defined sequence $\omega_v$. Furthermore, by
construction we obtain that $|D(1,j)|\in [0,M]$ whenever $j$ is not contained in
a repeated block.

In order to ensure that $\rho_{\overline{\cO_\sigma(\omega_v)}} \ssq v+\R v^\perp$,
we need to show that 
\begin{equation}
  \nLim \ntel \max\left\{ |D(i,j)| \mid |j-i|\leq n\right\} \ = \ 0 \ .
\end{equation}
Since the $a_n$ grow super-exponentially, this will be a direct implication of
the following.
\begin{prop}\label{p:is}
 If $0<j-i\leq a_n$, then $|D(i,j)| \leq 2nM+1$.
\end{prop}
For the proof, we need to introduce some further notation.  We say that $j \in
\N$ has \emph{depth} $d$, and write $\mathop\mathrm{depth}(j) = d$, if $d$ is
the maximal integer such that $j \in B_d$ and $B_1 \supsetneq B_2 \supsetneq
\dots \supsetneq B_d$ is a nested sequence of blocks with $\min B_i < \min
B_{i+1}$ and $\max B_i > \max B_{i+1}$ for all $i=1,\dots,d-1$. Note that the
nested sequence could be given by only one block $B_1=[1,a_n]$, but it always
exists since every integer is contained in some initial block. For the same
reason, $B_1$ will always be an initial block and $B_2$ is the largest repeated
block that contains $j$. Note also that the level of the blocks is decreasing,
and if $j \in [1, a_n]$ then $\mathop\mathrm{depth}(j) \leq n$. Moreover, if $n$
is the smallest integer such that $j\in [1,a_n]$, then $B_1$ is equal to $[1, a_n]$.

\begin{lemma}\label{l5}
  We have $|D(1,j)| \leq M\mathop\mathrm{depth}(j)$ for all $j\in\N$. In
  particular $|D(1,j)|\leq Mn$ for all $j \in [1, a_n]$.
\end{lemma}

\begin{proof}
  We prove the lemma for all $j \in [1, a_n]$ by induction on $n$. The statement
  holds for $j \in [1,a_1]$, since on this interval we apply rule II to all $j$
  and consequently $D(1,j)\in[0,M]$. Assume that the estimate holds for all $j
  \in [1, a_n]$ and let $j' \in [1, a_{n+1}]$. If $\mathop\mathrm{depth}(j') =
  1$, the statement holds by construction.  Note here that if $\omega_v(j)$ is
  chosen according to rule II, then $|D(1,j)| \in [0,M]$, whereas if we apply
  rule III then $|D(1,j)| \in [-M,M]$.

  Now, assume that $\mathop{depth}(j')=d$ and the block $B_2 = [m+1,
  m+a_k] \subsetneq [1, a_{n+1}]$ is of level $k \leq n$. Then $m$ is not
  contained in any block and thus $|D(1,m)| \leq M$ again by
  construction. Further, we have $\mathop{depth}(j'- m) \leq
  \mathop\mathrm{depth}(j') - 1$, and consequently
$$|D(m+1,j')| \ = \  |D(1, j'-m)|\  \leq \ M(\mathop\mathrm{depth}(j') - 1) \ .$$
Note that here $\omega_v(m+1),\dots,\omega_v(m+a_k) = \omega_v(1), \dots, \omega_v(a_k)$. Together,
we obtain
$$|D(1,j')| \ \leq \ |D(1,m)| + |D(m+1,j')| \ \leq \ C\mathop\mathrm{depth}(j')$$
as required.
\end{proof}

\begin{proof}[Proof of Proposition \ref{p:is}.]
  Fix $k > 0$. We proceed again by induction on $n$. Assume that the statement
  holds for all $i,j \in [1, a_{l}]$, $l \leq n$, and suppose that $i,j \in [1,
  a_{n+1}]$ with $|D(i,j)| > 2Mk+1$. We have to show that $j-i > a_k$.

  If $i, j$ are both contained in a repeated block $B = [m+1, m+a_{p}]$ of level
  $p \leq n$, then $|D(i,j)| = |D(i-m, j-m)|$ and the induction statement
  applies. Thus, we may assume that this does not happen. Due to Lemma~\ref{l5}
  we have
$$|D(i,j)| \ \leq \ |D(1,j)| + |D(1,i-1)| \ \leq \ M(\mathop{depth}(j) + \mathop{depth}(i-1)) \ ,$$
so that either $j$ or $i-1$ has depth bigger than $k$. We distinguish three
cases.

First, if both have depth bigger or equal to $k$, then as they cannot both be
contained in a single repeated block, they have to be contained in disjoints
blocks of level bigger than or equal to $k$. However, as two such blocks are at
least $(d_{k}-1)a_k$ apart, the statement follows.

Secondly, assume that $d = \mathop{depth}(i-1) > k$ and $\mathop{depth}(j) \leq
k$.  Let $B_1 \supsetneq \dots \supsetneq B_d$ be a nested sequence of blocks as
in the definition of $\mathop{depth}(i-1)$, with $B_1 = [1, a_{n+1}]$ and $B_2 =
[m+1 ,m+a_p]$.  Since $\mathop{depth}(i-1) > k$, we have $p \geq k$.  In the
case $i-1 \notin [m + a_p - a_k + 1, m + a_p]$, we have $j - i > a_k$ as
required.  Otherwise, we have that
\begin{eqnarray*} \lefteqn{|D(i, m + a_p)| \ = \
 | D(i - m - a_p + a_k, a_k)|} \\ & = & |D(1,
  a_k)| +|D(1, i - m - a_p + a_k - 1)| \ \leq \ Mk+1 \ , 
\end{eqnarray*}
using $D(1,a_k)\in [0,M]$ and Lemma~\ref{l5}.  Consequently, we obtain
\begin{eqnarray*}
  |D(i,j)| & \leq & |D(i, m + a_s) + D(m + a_s + 1, j)|
 \\ & = & |D(i, m + a_s) - D(1, m + a_s) + D(1, j)| \\
& \leq & Mk + 1 + M\mathop\mathrm{depth}(j) \ \leq \ 2Mk + 1 \ ,
\end{eqnarray*}
contradicting our assumption.  Finally, the case $\mathop{depth}(i-1) \leq k$
and $\mathop\mathrm{depth}(j) > k$ can be treated in an analogous way.
\end{proof}

As mentioned above, Proposition~\ref{p:is} implies that $\rho_{\cM(\omega_v)}
\ssq v+\R v^\perp$, and if we let $M_v=h_{\cR}^{-1}(\cM(\omega_v))$, then
according to Corollary \ref{c.symbolic_rotset} the same will be true for the
rotation set $\rho_{M_v}(F)$. It remains to show that $\rho_{\cM(\omega_v)}$ is
a segment of positive length. To that end, we note that for for all $n\in\N$ and
$j\in[1,a_n]\cap C_n$ we have $\omega_v(j)\in\{0,\iota\}$, where $\iota=1$ if
$n$ is odd and $\iota=2$ if $n$ is even. In the first case, (F4) implies that
the fraction of $2$'s in the interval $[1,a_n]$ is bounded by
$\delta=\frac{\|v\|}{10\max\{\alpha,\beta\}}$. At the same time, the requirement
that $D(1,a_n)\in[0,M]$ implies that a proportion of
$\|v\|/\max\{\alpha,\beta\}$ of symbols in $[1,a_n]$ must be non-zero. This
yields that the frequency of 1's in $[1,a_n]$ is greater than $9\delta$. For
even $n$, we obtain exactly the opposite estimates for the frequencies of $1$'s
and $2$'s. In the limit $n\to\infty$, this yields the existence of two distinct
vectors in $\rho_{\cM(\omega_v)}$. This completes the proof of
Theorem~\ref{t.mainresult}(a).

\subsection{Plane separating continua}

For the construction we make again use of the general block structure presented
above, this time with the following specifications.
\romanlist
\item We choose $(d_n)_{n\in\N}$ so that all $d_n$ are even and
  $\delta_{\infty}\leq \frac{1}{32}$.
\item We choose integers $K\geq 17$ and $L\geq 64$ and let $a_1 = (3L+ 4)K$
  and $b_n = (3L + 4)K$ for all $n\in\N$. The sequence $\nfolge{a_n}$ is then defined
  inductively by $a_{n+1}=(b_nd_n+1)a_n$, according to the general block
  structure introduced above.
\item Note that due to the choice of $b_n$ we have $a_{n+1}\geq 8a_nb_n$ for all
  $n\in\N$, which implies in particular that $\sum_{j=1}^{n-1}a_jd_j\leq
  a_nd_n/2$.  \listend

  Then we construct $\omega = (\omega_n)_{n\in\N}$ inductively on the sets $\cA_{n}$ as
  follows. Suppose $\omega_j$ is defined for all $j \in [1, a_n]$, and hence for all
  $j \in \cA_n$ (recall $\omega_{j+ka_nd_n} = \omega_j$ for all $j \in [1, a_n]$ and $k \in
  \N$). We extend the definition to $[1, a_{n+1}]$, and thus to $\cA_{n+1}$, as
  follows. Let
\begin{equation}
 \begin{split}
   p_n & = (L+1)Ka_nd_n - a_nd_n +1 +\sum_{j=1}^{n} a_jd_j/2 \ , \\ q_n & = \
   (L+1)Ka_nd_n + a_nd_n - \sum_{j=1}^{n} a_jd_j/2 - 1 \ .
\end{split}
\end{equation}
Then divide $[1, a_{n+1}]$ into the following seven intervals (see Figure
\ref{fig:intervals1}).
\begin{itemize}
\item[] $I^0_1=[1, (LKd_n+1)a_n]$,
\item[] $I^1_1=[(LKd_n+ 1)a_n+1,p_n - 1]$,
\item[] $I_1^2=[p_n, q_n]$,
\item[] $I^1_2 = [q_n + 1, (L + 2)Kd_na_n]$,
\item[] $I^0_2=[(L + 2)K d_na_n + 1, ((2L + 2)Kd_n + 1)a_n]$,
\item[] $I^2_2 = [((2L + 2)Kd_n + 1)a_n+1, (2L + 4)K d_na_n]$,
\item[] $I^0_3=[(2L + 4)K d_na_n +1, a_{n+1}]$.
\end{itemize}
Due to the choice of $p_n$ and $q_n$, the following properties are easy to
verify.

\begin{itemize}
\item[(PQ1)] The intervals $I^0_1,I^0_2$ and $I^0_3$ all have the same length
  $(LKd_n+1)a_n$ and start and end with a block of level $n$ (and thus with
  blocks of all levels $k\leq n$).
\item[(PQ2)] The length of $I^1_1$ and $I^1_2$ is between $(K-1)a_nd_n$ and
  $Ka_nd_n$.\\
{\em (Note here that due to (iii) we have $\sum_{j=1}^{n-1} a_jd_j/2 \leq a_nd_n/4$.) }
\item[(PQ3)] The length of $I_1^2$ is between $a_nd_n/2$ and $a_nd_n$.
\item[(PQ4)] The interval $I_1^2$ is concentric around a block $B_n$ of level $n$. 
\item[(PQ5)] The distance of $p_n$ and $q_n$ to any block $B_k$ of level $k\leq
  n$ is at least $a_kd_k/4$.\\
  {\em (In order to check this for $p_n$, note that for each $k \leq n$ a block
    of level $k$ starts at $(L+1)Ka_nd_n-a_nd_n+1+\sum_{j=k+1}^n a_jd_j/2$ and
    $\sum_{j=1}^k a_jd_j/2 \leq 3a_kd_k/4$. A similar comment applies to $q_n$.) }
\end{itemize} 
Let $I^0 = I^0_1 \cup I^0_2 \cup I^0_3$, $I^1 = I^1_1 \cup I^1_2$, $I^2 = I_1^2
\cup I^2_2$ and $I^* = I^1_1 \cup I_1^2 \cup I^1_2$.  
\begin{itemize}
\item[(PQ6)] The intervals $I^*$ and $I^2_2$ both have length $(2Kd_n-1)a_n$.
\end{itemize}

\begin{figure}[ht]
\unitlength=2.5mm
\begin{picture}(15,9)(15,2)
\put(-8,6){\line(1,0){64}}
\put(-8,5.5){\line(0,1){1}}
\put(8,5.5){\line(0,1){1}}
\put(0,7){$I^0_1$}
\put(9.5,7){$I_1^1$}
\put(12,7){$I_1^2$}
\put(14,7){$I^1_2$}
\put(12,5.5){\line(0,1){1}}
\put(23,7){$I_2^0$}
\put(13,5.5){\line(0,1){1}}
\put(35,7){$I_2^2$}
\put(47,7){$I_3^0$}
\put(16,5.5){\line(0,1){1}}
\put(32,5.5){\line(0,1){1}}
\put(40,5.5){\line(0,1){1}}
\put(56,5.5){\line(0,1){1}}
\put(-8,5){$\underbrace{\qquad \qquad \quad \quad \quad \quad \quad \quad \quad \, }_{(LKd_n+1)a_n}$}
\put(8.2,5){$\underbrace{\qquad \quad \quad \quad \ }_{(2Kd_n-1)a_n}$}
\put(16.2,5){$\underbrace{\qquad \qquad \quad \quad \quad \quad \quad \quad \quad }_{(LKd_n+1)a_n}$}
\put(32.2,5){$\underbrace{\quad \quad \quad \quad \quad \  }_{(2Kd_n-1)a_n}$}
\put(40.2,5){$\underbrace{\qquad \qquad \quad \quad \quad \quad \quad \quad \quad }_{(LKd_n+1)a_n}$}
\put(8.2,9){$\overbrace{\qquad \quad \quad \quad\  }^{I^*}$}
\end{picture}
\caption{The configuration of intervals.
\label{fig:intervals1}}
\end{figure}
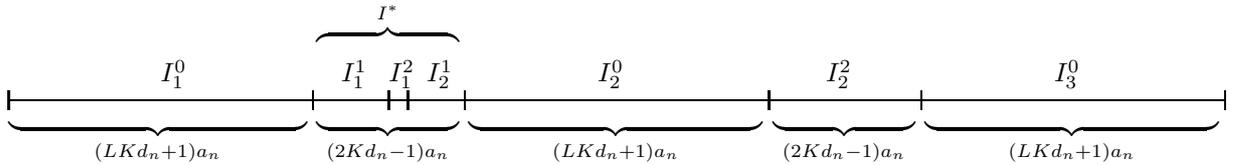

We define
\begin{equation} \label{e.plane-separating_def}
\omega_j = \ \begin{cases}
0 \ \mbox{ if } \ j\in I_0\setminus \cB_n\\
1 \ \mbox{ if } \ j\in I_1\setminus \cB_n\\
2 \ \mbox{ if } \ j\in I_2\setminus \cB_n\\
\end{cases}
\end{equation}
for all $j \in [1,a_n] \setminus \cB_n$ and $\omega_{j+ka_nd_n} = \omega_j$ for all $j \in
[1, a_n]$ and $k \in \N$. By induction on $n\in\N$ this yields a sequence
$\omega = (\omega_j)_{j \in \N}$, which follows our general block structure
introduced above and is, in particular, Toeplitz. \medskip

Recall that $v_0 = (0, 0)$, $v_1 = (1, 0)$ and $v_2 = (0, 1)$ are the integer
vectors associated to the partition and for every interval $J\ssq \N$ we
write $\rho(J) = \frac{\psi(J)}{|J|}$, where $\psi(J) = \sum_{j \in J}v_{\omega_j}$.

\begin{lemma}\label{l2}
$\rho([1, a_{n+1}]) \in B_{\frac{1}{8}}(v_0)$ for all $n \in \N_0$.
\end{lemma}
\begin{proof}
  With the above notions, we have that $\{ j \in [1,a_{n+1}]: \omega_j \neq 0 \}
  \subseteq ([1, a_{n+1}]\cap \cB_n) \cup I^1 \cup I^2$.  By (F4), we have
  $|[1,a_{n+1}]\cap \cB_n| \leq \delta_{\infty}a_{n+1} \le
  \frac{1}{16}a_{n+1}$. Further, we have that $|I^1 \cup I^2| \leq 
  \frac{4}{3l}a_{n+1} \le \frac{1}{16}a_{n+1}$.  The statement follows.
\end{proof}

\begin{lemma}\label{l3} Suppose  $J = [1, m]$ or $J = [m, a_{n+1}]$ for some $n\in\N_0$ and $m\in(1,a_{n+1})$.
  Then $\rho(J) \in B_{\frac{1}{8}}(v_0)$.
\end{lemma}
\begin{proof}
  We proceed by induction on $n$. Assume that the statement holds for
  $n\in\N_0$. Let $ J= [1, m] \subset [1, a_{n+1}]$ (the proof for $J = [m,
  a_{n+1}]$ is similar). We distinguish several cases.\smallskip

  {\bf Case 1:} $J \subset I^0_1$. Let 
$$J \cap \cB_n = B_1 \cup \dots \cup B_j \cup B',$$
where $B' = B_{j+1} \cap J$ and $B_i$, $i = 1, \dots, j+1$, are full blocks of
level $\le n$.  By the previous lemma $\rho(B_i) \in B_{\frac{1}{8}}(v_0)$ for
$i = 1, \dots , j$. Due to the induction hypothesis $\rho(B') \in
B_{\frac{1}{8}}(v_0)$. Hence $\rho(J \cap \cB_n) \in B_{\frac{1}{8}}(v_0)$.  Since
by construction $\rho(J \setminus \cB_n) = v_0$, we have $\rho(J) \in
B_{\frac{1}{8}}(v_0)$.\smallskip

{\bf Case 2:} $J\subseteq I^0_1 \cup I^*$ but $J \not\subseteq I^0_1$. By
(F4) we have that $|I^0_1\cap\cB_n| \leq \frac{1}{16} |I^0_1|$, and hence
$\rho(I^0_1) \in B_{\frac{1}{16}}(v_0)$.  Further $|J \cap I^*| \leq
|I^*| \leq \frac{2}{l} |I^0_1| <
\frac{1}{16} |I^0_1|$.  Thus $\rho(J) \in B_{\frac{1}{8}}(v_0)$.\smallskip

The remaining cases $J \subseteq I^0_1 \cup I^* \cup I^0_2$, $J \subseteq I^0_1
\cup I^* \cup I^0_2 \cup I^2_2$ and $J \subseteq [1,a_{n+1}]$ can be treated by
similar arguments.
\end{proof}

\begin{lemma} \label{l.long_intervals} Suppose that $J\ssq I^i_j$ is an interval
  of length $\geq a_nd_n/2$, where $i=0,1,2$ and $j=1,2$ or $i=0$ and
  $j=3$. Then $\rho(J)\in B_{\frac{1}{16}}(v_i)$.
\end{lemma}
\begin{proof}
  By (F4), we obtain that $|J\cap \cB_n|/|J| \leq \frac{1}{16}$. Since all free
  positions in $J$ are filled by $i$'s in the $(n+1)$st step of the construction,
  this implies the statement.
\end{proof}

\begin{cor}\label{l4} For $i=0,1,2,j=1,2$ or $i=0,j=3$ we have $\rho(I^i_j)\in
  B_{\frac{1}{16}}(v_i)$.
\end{cor}

\begin{lemma}\label{l5} Let $J\ssq I\ssq \N$ be intervals and assume that $J$
  has one endpoint in common with $I$. 

\begin{itemize}
\item[(a)] If $I=I^0_j$, $j=1,2,3$, then $\rho(J)\in B_{\achtel}(v_0)$. 
\item[(b)] If $I=I^1_j$, $j=1,2$, then $\rho(J)\in B_{\achtel}(v_1)$.
\item[(c)] If $I=I^2_j$, $j=1,2$, then $\rho(J)\in B_{\achtel}(v_2)$.
\item[(d)] If $I=I^*$, then $\rho(J)\in B_{\achtel}(v_1)$. 
\end{itemize}
\end{lemma}

\begin{proof}
\begin{itemize}
\item[\em (a)] If $J\ssq I^0_1$ starts with $1$ or $J\ssq I^0_3$ ends with
  $a_{n+1}$, then the statement is contained in Lemma \ref{l3}. However, by
  construction the configurations of symbols in the intervals $I^0_1,\ I^0_2$
  and $I^0_3$ are identical. In order to see this, note that since these have
  the same length and all start and end with a block of level $n$ by (PQ1), the
  configuration of the blocks is identical, and all positions not contained in
  previous blocks are filled by $0$'s. Hence, by symmetry the statements for
  $I^0_1$ and $I^0_3$ extend to the other intervals.
\item[\em (b)] {\em and (c)} \quad The proofs of all cases of (b) and (c) are
  similar. Hence, we consider only one of them. The crucial observation is the
  fact that all endpoints of these intervals have distance $\geq a_kd_k/4$ to
  any block of level $k$. For the points $p_n$ and $q_n$, this is true by
  construction, see (PQ5). For endpoints of $I^1_1, I^1_2$ and $I^2_2$ this
  holds since the adjacent intervals $I^0_j$, $j=1,2,3$, start and end with
  blocks of level $n$. Hence, the nearest block of any level $k\leq n$ in one of
  the considered intervals can appear at distance $(d_k-1)a_k$ to the boundary
  points.

  Assume that $I=I^2_1$ and $J=[p_n,k]$.  If $J$ does not intersect any blocks
  of level $k\leq n$, then $\rho(J)=v_2$. Otherwise, let $m$ be the largest
  integer such that $J \cap B \ne \emptyset$ for some $m$-block $B$. Then by
  (PQ5) the length of $J$ is at least $a_md_m/4$, and due to (F4) we obtain that
  $|J\cap \cB_n|/|J|\leq 4\delta_\infty\leq\achtel$, which implies $\rho(J) \in
  B_{\frac{1}{8}}(v_2)$. As mentioned, the other cases are analogous.
\item[\em (d)] If $J$ is either contained in $I^1_1$ or $I^1_2$, then the
  statement is contained in (b). Otherwise, it follows from the fact that
  $\rho(I^1_j)\in B_{\frac{1}{16}}(v_1)$ by Corollary~\ref{l4} and
  $|I^2_1|/|I^1_j| \leq \frac{1}{K-1} \leq \frac{1}{16}$.
\end{itemize}
\end{proof}

\begin{prop}\label{p:sp}
  Let $T = \{ \lambda v_i + (1 - \lambda )v_j : i,j \in \{ 0, 1, 2 \} , \lambda
  \in [0, 1] \}$ and $S = B_{\frac{1}{8}}(T)$.  Then $\rho(J) \in S$ for every
  $J = [a, b] \subset \N$.
\end{prop}

\begin{proof}
  Let $n$ be the smallest integer such that $J$ is contained in a block of level
  $n+1$. We prove the statement by induction on $n$ and may thus assume that $J$
  is not entirely contained in any block of level $k\leq n$. Moreover, since the
  structure inside all blocks of level $n+1$ is the same, we may assume without
  loss of generality that $J \subseteq [1, a_{n+1}]$.  We distinguish several
  cases.\smallskip

  {\bf Case 1.} Suppose that $J$ intersects both $I^*$ and
  $I^2_2$. In this case $J$ contains $I^0_2$, and by
  Corollary~\ref{l4} we have $\rho(I^0_2)\in
  B_{\frac{1}{16}}(v_0)$. Moreover, $|I^*\cup I^2_2|/|I^0_2| \leq
  \frac{4}{3l} \leq 1/16$. If $J$ also intersects the intervals $I^0_1$
  and $I^0_3$, say $J'=J\cap I^0_1$ and $J''=J\cap I^0_3$, then
  $\rho(J'),\rho(J'')\in B_{1/8}(v_0)$ by Lemma~\ref{l5}(a).
  Putting everything together, we obtain $\rho(J)\ssq B_{1/8}(v_0)$.\smallskip

  {\bf Case 2.} Suppose that $J$ intersects exactly two of the five intervals
  $I^0_1,I^*,I^0_2,I^2_2,I^0_3$. Since all the subcases are similar, we only
  treat one and assume $J$ intersects $I^0_1$ and $I^*$. Let $J'=J\cap I^0_1$
  and $J''=J\cap I^*$. Then $\rho(J')\in B_{\achtel}(v_0)$ by \ref{l5}(a),
  whereas $\rho(J'')\in B_{\achtel}(v_1)$ by Lemma~\ref{l5}(d). Consequently,
  $\rho(J)$ is a convex combination of a vector in $B_{\achtel}(v_0)$ and a
  vector in $B_{\achtel}(v_1)$, and therefore belongs to $S$.

  {\bf Case 3.} Suppose $J \subseteq I^*$ intersects at least two of the
  intervals $I^1_1,I_1^2$ and $I^1_2$. Let $J=J'\cup J''\cup J'''$, where
  $J'=J\cap I^1_1$, $J''=J\cap I_1^2$ and $J'''=J\cap I^1_2$. Then
  $\rho(J'),\rho(J'')\in B_{\achtel}(v_1)$ by Lemma~\ref{l5}(b), and
  $\rho(J'')\in B_{\achtel}(v_2)$ by Lemma~\ref{l5}(c). Hence, we obtain again
  that $\rho(J)\in S$. 

  {\bf Case 4.} Finally, suppose that $J$ is contained in just one of the seven
  intervals of the decomposition of $[1,a_{n+1}]$, say $J\ssq I^i_j$. Then 
  \[
      J\cap \cB_n \ = \ B'\cup B_1 \cup \ldots \cup B_m \cup B'' \ ,
  \]
  where $B'=J\cap B_0$, $B''=J\cap B_{m+1}$ and the $B_l$ with $l=0\ld m+1$ are
  those maximal blocks contained in $\cB_n$ which intersect $J$, ordered in an
  increasing way. Since $J$ is not entirely contained in one block, we can use
  Lemmas~\ref{l2} and \ref{l3} to see that $\rho(B'),\rho(B_1)\ld$ $
  \rho(B_m),\rho(B'')\in B_{\achtel}(v_0)$, and hence $\rho(J\cap \cB)\in
  B_{\achtel}(v_0)$. At the same time we have $\rho(J\smin \cB)=v_i$, such that
  again $\rho(J)$ is contained in $S$. 
\end{proof}

\begin{prop}\label{p:cc}
$\rho_{\Cl(\mathcal{O}(\omega_{\textrm{sp}}, \sigma))}$ separates the plane.
\end{prop}

\begin{proof} We have that
$$\rho_{\Cl(\mathcal{O}(\omega_{\textrm{sp}}, \sigma))}
= \bigcap_{k\in\N} \overline{\bigcup_{n\geq k} k}$$ where $k=\{\rho(J)
  \mid J\ssq \N \mbox{ is a finite interval with } |J| = n\}$.

  Given $0\leq i< j\leq 2$ we let $S_{ij}=\{ \lambda v_i + (1 - \lambda )v_j :
  \lambda \in [0, 1] \}$. For $n\in\N$, let $J_1=I_1^2=[p_n,q_n]$ and choose an
  interval $J_2\ssq I^2_2$ which has the same length as $J_1$ and is concentric
  around a block of level $n$. Since this also holds for $J_1$, we have that the
  configuration of blocks inside both intervals is the same. Since free
  positions in both intervals are both filled by $2$'s, we have that
  $\rho(J_1)=\rho(J_2)\in B_{\achtel}(v_2)$.

  Let $M_n\in\N$ be such that $J_2=J_1+M_n=\{j+M_n\mid j\in J_1\}$ and let
  $\rho^n_i=\rho(J_1+i)$ for $i=0\ld M_n$. Then all the $\rho^n_i$ belong to
  $S$, and the distance between $\rho^n_i$ and $\rho^n_{i+1}$ is at most
  $2/|J_1|$. Moreover, as $i$ increases from $0$ to $M_n$, the interval $J_1+i$
  will leave $I_1^2$ in order to enter $I^1_2$, move on to $I^0_2$ and
  eventually enter $I^2_2$ until it stops at $J_1+M_n=J_2$.

According to Lemmas~\ref{l.long_intervals} and \ref{l5}, the corresponding
vectors $\rho^n_i$ always remain in $S$. Further, they start in
$B_\achtel(v_2)$, then move to $B_\achtel(v_1)$ while remaining in $S_{12}$,
then move to $B_\achtel(v_0)$ while remaining in $S_{01}$ and finally return to
$B_\achtel(v_2)$ while remaining in $S_{02}$. Note here that $|J_1|\geq a_nd_n/2$,
such that Lemma~\ref{l.long_intervals} applies whenever $J_1+i$ is entirely
contained in one of the intervals of the decomposition, and otherwise we can
always combine two of the statements of Lemma~\ref{l5}

Since $|J_1|\nearrow\infty$ as $n\to\infty$, it follows easily from these
facts that the upper Hausdorff limit of the sequence of finite sets
$\{\rho^n_0\ld \rho^n_{M_n}\}\ssq k$ contains a continuum $\cC$ that separates
the two connected components of the complement of $S$. Since
$\cC\ssq\rho_{\Cl(\mathcal{O}(\omega_{\textrm{sp}}, \sigma))}$, this completes
the proof.
\end{proof}

This shows Theorem~\ref{t.mainresult}(b).

\subsection{Non-empty interior}

It remains to construct $\omega'\in\Sigma^+$ such that $\rho_{\cM(\omega')}$ has
non-empty interior. It turns out that in comparison with the previous cases this
is quite easy.  We use the same block construction as before, with $b_n=1$ for
all $n\in\N$ and $d_n$ chosen such that $\delta = \delta_\infty<
1/10$. Let $\Delta_\delta=\{sv_1+tv_2\mid  s,t>\delta,\ 
s+t<1-\delta\}$ and choose a sequence of vectors $\rho_n\in\Delta_\delta$ such
that the coordinates of $\rho_n$ are integer multiples of $1/a_n$ and $\{\rho_n\mid
n\in\N\}$ is dense in $\Delta_\delta$. Then we simply define $\omega'$
inductively on $[1,a_n]$ in such a way that $\frac{1}{a_n}\sum_{i=1}^{a_n}
v_{\omega'(i)} = \rho_n$ for all $n\in\N$. This is possible, since in each stage
of the construction we have $|[1,a_n]\smin\cB_{n-1}|=a_n-|[1,a_n]\cap \cB_n|
\geq (1-\delta_n)a_n$.  We thus obtain that $\Delta_\delta\ssq
\bigcap_{n\in\N} \overline{\{\rho_k\mid k\geq n\}} \ssq
\rho_{\cM(\omega')}$. \medskip

This proves Theorem~\ref{t.mainresult}(c) and thus completes the proof of
Theorem~\ref{t.mainresult}.

\section{Abundance.}\label{Abundance}

Finally, in this section we want to prove that the phenomena given by
Theorem~\ref{t.mainresult} are abundant, in the sense that they occur for an
open set in $\homeo_0(\T^2)$. Recall that we denote $\cF$ the set of those
$f\in\homtwo$ having non-empty interior rotation set. The result we want to
prove is the following.
\begin{thm}
  The family $\mathcal{G}$ in Theorem~\ref{t.mainresult} contains an open and dense set of
  $\mathcal{F}$.
\end{thm}
This statement essentially follows from series of results on Axiom A
diffeomorphisms which is already collected in \cite{PasseggiRotSetAxA}. We
mainly follow that paper and keep the exposition brief. Recall that
$f\in\homeo_0(\T^2)$ is an axiom A diffeomorphism if the non-wandering set is
hyperbolic and contains a dense set of periodic points. We call by
$\mathcal{F}_0\subset \mathcal{F}$ the set of those axiom A maps having
zero-dimensional (totally disconnected) non-wandering set. The elements of
$\mathcal{F}_0$ are called fitted Axiom A.
\begin{thm}[\cite{ShubSullHomoandDynSys}]
  The set $\mathcal{F}_0$ is dense in $\mathcal{F}$.
\end{thm}

\begin{thm}[\cite{NitekiSemStabDiff}]
For any $f\in \mathcal{F}_0$ there is a $C^0$-neighborhood $\mathcal{U}(f)$ of $f$ so that
for all $g\in\mathcal{U}(f)$ there exists a semiconjugacy $h$ between $g$ and $f$, that is,
a continuous onto map $h$ so that $h\circ g=f\circ h$. Moreover, the semiconjugacy can be
chosen in the homotopy class of the identity.
\end{thm}

The last theorem implies in particular that given $g\in\mathcal{U}(f)$ as above,
we have $\rho_{C}(G)=\rho_{h(C)}(F)$ for any closed invariant set $C$ of $g$. Thus if
we prove that $\mathcal{F}_0\subset \mathcal{G}$, we automatically have that
$$\bigcup_{f\in\mathcal{A}_0}\mathcal{U}(f)\subset \mathcal{G}$$
by means of the last theorem, which proves Theorem~\ref{t.genericity}. Thus, the
remainder of this section is devoted to showing that $\mathcal{F}_0\subset
\mathcal{G}$.  \medskip

Recall that a {\em basic piece} $\Lambda \subset \T^2$ of a diffeomorphism $f$
is a hyperbolic transitive set which is locally maximal. Given a set $X\subset
\R^2$ we denote its convex hull by $\textrm{conv}(X)$.  The proof of the
following statement can be found in \cite[Corollary 5.2]{PasseggiRotSetAxA}.

\begin{thm}\label{lambdarot}
  Every $f\in\mathcal{F}_0$ has a basic piece $\Lambda$ so that
  $\textrm{\conv}(\rho_\Lambda(F))$ has non-empty interior.
\end{thm}

We denote the basic set given in the last theorem by
$\Lambda_{\textrm{rot}}$. The following results is a mixture of \cite[Lemma
5.2]{PasseggiRotSetAxA} and the proof of \cite[Theorem 5.2]{PasseggiRotSetAxA}.
There is only one new consideration which is not done in
\cite{PasseggiRotSetAxA}, which is the following.

In \cite[Lemma 5.2]{PasseggiRotSetAxA} the assumption is that we have a basic
piece whose rotation set is not a single point, and the conclusion is the
existence of a heteroclinic relation for its lift between a fixed point and an
integer translation of it. In our situation given by Theorem~\ref{lambdarot}, we
have a basic piece whose rotation set has at least two non-colinear vectors, and
the conclusion we need is the existence of a fixed point of the lift that has
heteroclinic relations with two non-colinear integer translations of itself.
However, the proof for this case is completely analogous to that of \cite[Lemma
5.2]{PasseggiRotSetAxA}. Then, applying the same arguments as in the proof of
\cite[Theorem 5.2]{PasseggiRotSetAxA}, one easily obtains the following result.

\begin{thm}\label{partitions}
  Suppose $f\in\mathcal{F}_0$ and let $\Lambda_{\textrm{rot}}$ be the basic
  piece given by Theorem~\ref{lambdarot}.  Then there exists a positive integer
  $n$ and an invariant set $\Lambda\subset \Lambda_{\textrm{rot}}$ which admits
  a Markov partition $\mathcal{R}=\{R_0,R_1,R_2\}$, so that $\bigcup_{i=0}^2R_i$
  is contained in a topological disk $D$, $f_{|\Lambda}$ is conjugated via
  $h_{\mathcal{R}}$ to the full shift on $\{0,1,2\}^{\Z}$, and if we consider
  lifts $\tilde{R}_0,\tilde{R_1},\tilde{R}_2$ of $R_0,R_1,R_2$ we have:
\begin{itemize}
\item $F^n(\tilde{R}_0)\cap \tilde{R}_0\neq\emptyset,$
\item $F^n(\tilde{R}_1)\cap \tilde{R}_1+v\neq\emptyset,$
\item $F^n(\tilde{R}_2)\cap \tilde{R}_2+w\neq\emptyset,$
\end{itemize}
where $v,w\in \Z^2\setminus\{0\}$ are non-colinear.

\end{thm}

This implies that any element $f$ in $\mathcal{F}_0$ has a power which already
has very similar properties to the ones we used in the constructions in the
previous sections. In fact, by consindering an iterate $f^n$ of $f$ and
performing a linear change of coordinates on \torus, we may assume that $v=v_1$
and $w=v_2$ (see \cite{kwapisz:1992} for details). Therefore $f^n$ has minimal
sets with the rotation sets described in Theorem~\ref{t.mainresult}. However,
since $\rho_M(F^n)=n\rho_M(F)$, the same applies to $f$ itself. This completes the
proof of Theorem~\ref{t.genericity}.

\enlargethispage*{1000pt}


\end{document}